\documentclass{amsart}
\usepackage{graphicx}
\usepackage{amsmath,amssymb}
\usepackage{amsthm}
\usepackage{color}

\newtheorem{thm}{Theorem}[section]
\newtheorem{cor}[thm]{Corollary}
\newtheorem{lem}[thm]{Lemma}

\theoremstyle{definition}

\theoremstyle{remark}
\newtheorem{rem}[thm]{Remark}
\theoremstyle{example}

\theoremstyle{conjecture}

\numberwithin{equation}{section}

\newcommand{\B}{{\mathbb B}}
\newcommand{\D}{{\mathbb D}}

\newcommand{\R}{{\mathbb R}}

\newcommand{\C}{{\mathbb C}}

\newcommand{\N}{{\mathbb N}}

\newcommand{\XX}{{\mathbb X}}
\newcommand{\YY}{{\mathbb Y}}

\newcommand{\calD}{{\mathcal D}}

\newcommand{\calW}{{\mathcal W}}

\makeatletter
\newcommand{\doublewidetilde}[1]{{%
  \mathpalette\double@widetilde{#1}%
}}
\newcommand{\double@widetilde}[2]{%
  \sbox\z@{$\m@th#1\widetilde{#2}$}%
  \ht\z@=.9\ht\z@
  \widetilde{\box\z@}%
}

\def\D{\mathbb{D}}
\def\R{\mathbb R}

\def\C{\mathbb{C}}

\def\msk{\medskip}

\def\bege{\begin{equation}} \def\ende{\end{equation}}
\def\begr{\begin{eqnarray}} \def\endr{\end{eqnarray}}

\makeatother

\begin{document}

\title[Weighted composition operator]{Difference of weighted composition operators on weighted-type spaces in the unit ball}%

\author{Bingyang Hu and Songxiao Li$^*$}%

\address{Bingyang Hu: Department of Mathematics, University of Wisconsin, Madison, WI 53706-1388, USA.}%
\email{bhu32@wisc.edu}

\address{Songxiao Li:  Institute of Fundamental and Frontier Sciences, University of Electronic Science and Technology of China, 610054, Chengdu, Sichuan, P.R. China.}
\address{Institute System Engineering, Macau University of Science and Technology,  Avenida Wai Long, Taipa, Macau. }%
\email{jyulsx@163.com}

\subjclass[2010]{32A37, 47B38}%
\begin{abstract}   In this paper, a new characterization is provided  for the boundedness, compactness and essential norm of the difference of two weighted composition operators on weighted-type spaces in the unit ball of $\C^n$.
\thanks{*Corresponding author.}
\vskip 3mm \noindent{\it Keywords}: Weighted composition operators, difference, weighted-type space.
\thanks{This  project was partially supported by the Macao Science and Technology Development Fund (No.186/2017/A3) and NSF of China (No.11471143 and No. 11720101003).   }
\end{abstract}



\maketitle

\section{Introduction}

 Let $\B$ be the open unit ball of $\C^n$ and $\partial \B$ the boundary of $\B$. For $a \in \B\backslash\{0\}$, the automorphism of $\B$ is defined by
$$
\Phi_a(z)=\frac{a-P_az-s_aQ_az}{1-\langle z, a \rangle}, \ z \in \B,
$$
where $s_a=\sqrt{1-|a|^2}$,
$$
P_a z=\frac{\langle z, a \rangle}{\langle a, a \rangle}a, \quad \textrm{and} \quad Q_az=z-P_az, \ z \in \B.
$$
When $a=0$, $\Phi_a(z)=-z$. For $z, w \in \B$, the  pseudo-hyperbolic distance  between $z$ and $w$ is given by
$$
\rho(z, w)=|\Phi_w(z)|.
$$
It is clear that $\rho(z, w) \le 1$, moreover, it is invariant under automorphism, that is,
$$
\rho(\phi(z), \phi(w))=\rho(z, w),
$$
for all $z, w \in\B$ and $\phi \in Aut(\B)$.

 Let   $H(\B)$ be the   space of all holomorphic functions on $\B$.  Let $\alpha>0$. An $f \in H(\B)$ is said to belong to the weighted-type space, denoted by $H_\alpha^\infty$, if
$$
\|f\|_{\alpha}:=\sup_{z \in \B} (1-|z|^2)^\alpha |f(z)|<\infty.
$$
It is well known that  $H_\alpha^\infty$ is a Banach space under the norm $\| \cdot\|_{\alpha}$.

Let $\varphi$ be a  holomorphic self-map of $\B$ and   $u \in H(\B)$. The  weighted composition operator $uC_\varphi: H(\B) \mapsto H(\B)$ is defined by
$$
uC_\varphi(f)(z)=u(z)  f ( \varphi(z) ) , \quad f \in H(\B), z \in \B.
$$
  Observe that $uC_\varphi(f)=M_u \circ C_\varphi (f)$, where $M_u(f)=uf$ is the  multiplication operator with symbol $u$ and $C_\varphi(f)=f \circ \varphi$ is the composition operator with symbol $\varphi$.

The boundedness and compactness of the operator $uC_\varphi$  are always important in the study of such operators (see, e.g.,  \cite{cm}). Recently, it is known that such properties can be merely captured by polynomials. More precisely, for a operator $uC_\varphi$ from $ \XX $ into $ \YY$, where $\XX$ and $\YY$ are some ``nice" analytic function spaces that are defined on the unit disc $\D$ (or $\B$, respectively),
\begin{enumerate}
\item[1.] $uC_\varphi: \XX \mapsto \YY$ is bounded if and only if
$$
\sup_{j\in \N} \frac{\|u \varphi^j\|_\YY}{\|z^j\|_\XX}<\infty \ \left(\textrm{or} \ \sup_{j \in \N}\sup_{\xi \in \partial \B} \frac{\|u\langle \varphi, \xi \rangle^j\|_\YY}{\|\langle z, \xi \rangle^j\|_\XX}<\infty, \ \textrm{respectively} \  \right);
$$
\item [2.] $uC_\varphi: \XX \mapsto \YY$ is compact if and only if $uC_\varphi: \XX \mapsto \YY$ is bounded and
$$
\limsup_{j \to \infty} \frac{\|u \varphi^j\|_\YY}{\|z^j\|_\XX}=0 \ \left(\textrm{or} \ \limsup_{j \to \infty}\sup_{\xi \in \partial \B} \frac{\|u\langle \varphi, \xi \rangle^j\|_\YY}{\|\langle z, \xi \rangle^j\|_\XX}=0, \ \textrm{respectively} \  \right).
$$
\end{enumerate}
Such a phenomenon was first found by Wulan, Zheng and Zhu in \cite{wzz} in the setting of the unit disk. They
  showed that $C_\varphi$ is compact on $\mathcal{B}(\D)$ if and only if $ \lim\limits_{j\to\infty} \|\varphi^j\|_{\mathcal{B}(\D)}=0$. Here $\mathcal{B}(\D)$ is the Bloch space, which consists of all analytic functions $f$ on $\D$ satisfying
 $\sup\limits_{z \in \D}(1-|z|^2) |f'(z)|<\infty. \nonumber$  In \cite{dai}, Dai  extended  the main result in \cite{wzz} to the unit ball. He showed that $C_\varphi$ is compact on $\mathcal{B}(\B)$ if and only if $ \lim\limits_{j\to\infty}\sup\limits_{\xi \in \partial \B} \|\langle \varphi, \xi \rangle^j \|_{\mathcal{B}(\B)}=0$.
 See \cite{co, dai, HLS, hlw, ll, SL, SL2, wzz, zhao}  for more results on such characterization of composition operators and weighted composition operators on some analytic function spaces.

Recently, the difference of composition operators (as well as the weighted composition operators) draws great attentions of lots of researchers, as it can be used to study the topological structure of the set of composition operators (as well as the weighted composition operators). For example, given $uC_\varphi$ and $vC_\psi$ two bounded operators acting from $\XX$ to $\YY$ as above, one may ask whether $uC_\varphi$ and $vC_\psi$ are in the same path component in $\calW(\XX, \YY)$,  the set of all  bounded weighted composition operators between $\XX$ and $\YY$, equipped with the topology induced by operator norm. More precisely, we are interested in that whether there exists a continuous mapping $\gamma: [0, 1] \mapsto \calW(\XX, \YY)$, such that $\gamma(0)=uC_\varphi$ and $\gamma(1)=vC_\psi$. In general, this is a hard question and it turns out that one should first understand the behavior of the difference of two weighted composition operators.

The line of this research was first started  by Berkson in \cite{Be}. In \cite{SL}, Shi and Li obtained several estimates for the essential norm of the difference of composition operators on $\mathcal{B}(\D)$.  Among others, they showed that
$$\|C_\varphi-C_\psi\|_{e, \mathcal{B}(\D)\rightarrow \mathcal{B}(\D)}\simeq  \lim_{j\to\infty}\|\varphi^j-\psi^j\|_{\mathcal{B}(\D)}.$$
 For further results of the difference under various  settings, we refer the readers to \cite{Be, cm, DO,  HLS, li, lw, Mo, Ni, S1, S2, Sj, SL, SL2} and the references therein.

In \cite{Ni}, Nieminen  obtained a characterization of the compactness of differences of  weighted composition operators on weighted-type spaces.
Motivated by the results in \cite{wzz} and \cite{Ni}, Hu, Li and Shi gave a new characterization for  the boundedness, compactness and  essential norm of the operator $uC_{\varphi}-vC_{\psi}: H^{\infty}_{\alpha} \to H^{\infty}_{\beta}$ in the unit disk.  More precisely,  they showed that $uC_{\varphi}-vC_{\psi}: H^{\infty}_{\alpha} (\D) \to H^{\infty}_{\beta}(\D)$ is bounded (respectively, compact) if and only if the sequence $\Big(\frac{\|u\varphi^j-v\psi^j\|_{\beta }}{\|z^j\|_{\alpha }} \Big)_{j=0}^\infty$ is bounded (respectively, convergent to 0 as $j\to\infty$).

 In this paper, we study the difference of two weighted composition operators between different weighted-type spaces in the unit ball, namely, the operator $uC_\varphi-vC_\psi: H^{\infty}_{\alpha} \to H^{\infty}_{\beta}$, where $u, v \in H(\B)$ and $\varphi, \psi$ are two holomorphic self-maps of $\B$. We characterize the boundedness, compactness and essential norm of the operator $uC_\varphi-vC_\psi$ by using $$
  \sup_{\xi \in \partial \B} \frac{\|u\langle \varphi, \xi \rangle^j-v\langle \psi, \xi \rangle^j\|_\beta}{\|\langle z, \xi \rangle^j\|_\alpha} $$
and
$$ \sup_{\xi, \xi' \in \partial \B}\frac{\|u\langle \varphi, \xi \rangle^j\langle \varphi, \xi' \rangle-v\langle \psi, \xi \rangle^j\langle \psi, \xi' \rangle\|_\beta}{\|\langle z, \xi \rangle^j\langle z, \xi' \rangle \|_\alpha}.$$
 For a non-polynomial description, we refer the readers to the paper \cite{DO} for details.

Throughout this paper, for $a, b \in \R$, $a \lesssim b$ ($a \gtrsim b$, respectively) means there exists a positive number $C$, which is independent of $a$ and $b$, such that $a \leq Cb$ ($ a \geq Cb$, respectively). Moreover, if both $a \lesssim b$ and $a \gtrsim b$ hold, then we say $a \simeq b$.
\bigskip

\section{Boundedness of $uC_\varphi-vC_\psi: H_\alpha^\infty \mapsto H_\beta^\infty$. }

In this section, we characterize the boundedness of the difference of weighted composition operators from $H^\infty_\alpha$ to $H^\infty_\beta$. For all $z, w \in \B$, define
$$
\flat_\alpha(z, w)=\sup_{\|f\|_{H_\alpha^\infty} \le 1} \left| (1-|z|^2)^\alpha f(z)-(1-|w|^2)^\alpha f(w) \right|.
$$
Let $\varphi$ and $\psi$ be holomorphic self-maps of $\B$, $u, v \in H(\B)$. We denote
$$
\calD_{u, \varphi}(z)=\frac{(1-|z|^2)^\beta u(z)}{(1-|\varphi(z)|^2)^\alpha},  \ \calD_{v, \psi}(z)=\frac{(1-|z|^2)^\beta v(z)}{(1-|\psi(z)|^2)^\alpha}.
$$
Moreover, for each $a \in \B$,  we define the following families of test functions on $\B$:
$$
f_a(z)=\frac{(1-|a|^2)^\alpha}{(1-\langle z, a\rangle)^{2\alpha}},
$$
\[ g_{\varphi, \psi, a}(z)=\begin{cases}
f_{\varphi(a)}(z) \cdot \frac{\langle \Phi_{\varphi(a)}(z), \Phi_{\varphi(a)}(\psi(a)) \rangle}{|\Phi_{\varphi(a)}(\psi(a))|}, & \Phi_{\varphi(a)}(\psi(a)) \neq 0; \\
0, & \Phi_{\varphi(a)}(\psi(a))=0,
\end{cases} \]
and
\[ g_{\psi, \varphi, a}(z)=\begin{cases}
f_{\psi(a)}(z) \cdot \frac{\langle \Phi_{\psi(a)}(z), \Phi_{\psi(a)}(\varphi(a)) \rangle}{|\Phi_{\psi(a)}(\varphi(a))|}, & \Phi_{\psi(a)}(\varphi(a)) \neq 0;\\
0, &  \Phi_{\psi(a)}(\varphi(a))=0.
\end{cases} \]
It is clear that $f_a$, $g_{\varphi, \psi, a}$ and $g_{\psi, \varphi, a}$ are holomorphic in $\B$ with $g_{\varphi, \psi, a}(\varphi(a))=0$ and $g_{\psi, \varphi, a}(\psi(a))=0$. Moreover, we have
$$
\sup_{a\in \B} \|f_a\|_\alpha\le 1, ~~~\sup_{a\in \B} \|g_{\varphi, \psi, a}\|_\alpha\le 1, ~~~\sup_{a\in \B} \|g_{\psi, \varphi, a}\|_\alpha \le 1.
$$

To state and prove our main results in this paper, we need some lemmas. The following well-known estimate can be found in \cite[Lemma 3.2]{DO}.\msk

\begin{lem} \label{lem01}
For $f \in H_\alpha^\infty$ and $z, w \in \B$,
$$
\left| (1-|z|^2)^\alpha f(z)-(1-|w|^2)^\alpha f(w) \right| \lesssim \|f\|_{\alpha} \rho(z, w).
$$
\end{lem}\msk

\begin{lem} \label{lem02}
Let $0<\alpha, \beta<\infty, u, v \in H(\B)$. Let further, $\varphi$ and $\psi$ be holomorphic self-maps of $\B$. Then the following inequalities hold:
\begin{enumerate}
\item[(i).] For each $a \in \B$,
\begin{eqnarray*}
|\calD_{u, \varphi}(a)| \rho(\varphi(a), \psi(a))
 \le  \|(uC_\varphi-vC_\psi)f_{\varphi(a)}\|_{\beta}  +\|(uC_\varphi-vC_\psi)g_{\varphi, \psi, a}\|_\beta;
\end{eqnarray*}
\item[(ii).] For each $a \in \B$,
\begin{eqnarray*}
|\calD_{v, \psi}(a)| \rho(\varphi(a), \psi(a)) \le  \|(uC_\varphi-vC_\psi)f_{\psi(a)}\|_{\beta} +\|(uC_\varphi-vC_\psi)g_{\psi, \varphi, a}\|_\beta.
\end{eqnarray*}
\item[(iii). ] For each $a \in \B$,
\begin{eqnarray*}
|\calD_{u, \varphi}(a)-\calD_{v, \psi}(a)|
&\lesssim& \|(uC_\varphi-vC_\psi)f_{\varphi(a)}\|_{\beta} +\|(uC_\varphi-vC_\psi)f_{\psi(a)}\|_\beta  \\
&& + \min\{ \|(uC_\varphi-vC_\psi) g_{\varphi, \psi, a}\|_{\beta}, \|(uC_\varphi-vC_\psi)g_{\psi, \varphi, a}\|_\beta\}.
\end{eqnarray*}
\end{enumerate}
\end{lem}

\begin{proof}
The idea of the proof follows from  \cite[Lemma 2.1]{HLS}.

(i). We only consider the case that $\Phi_{\varphi(a)}(\psi(a)) \neq 0$. Otherwise,
$$
|\calD_{u, \varphi}(a)| \rho(\varphi(a), \psi(a))|=0
$$
and hence there is nothing to prove. Then, by a simple calculation, we have
$$
\|(uC_\varphi-vC_\psi)f_{\varphi(a)}\|_\beta \ge |\calD_{u, \varphi}(a)|-\frac{(1-|\varphi(a)|^2)^\alpha(1-|\psi(a)|^2)^\alpha}{|1-\langle \psi(a), \varphi(a) \rangle|^{2\alpha}}|\calD_{v, \psi}(a)|
$$
and
$$
\|(uC_\varphi-vC_\psi)g_{\varphi, \psi, a}\|_\beta \ge \frac{(1-|\varphi(a)|^2)^\alpha (1-|\psi(a)|^2)^\alpha}{|1-\langle \psi(a), \varphi(a) \rangle|^{2\alpha}} |\calD_{v, \psi}(a)| \rho(\varphi(a), \psi(a)).
$$
Hence
\begin{eqnarray*}
&& |\calD_{u, \varphi}(a)| \rho(\varphi(a), \psi(a)) \\
&\le& \|(uC_\varphi-vC_\psi)f_{\varphi(a)}\|_\beta \cdot \rho(\varphi(a), \psi(a))    +\|(uC_\varphi-vC_\psi)g_{\varphi, \psi, a}\|_\beta \\
&\le& \|(uC_\varphi-vC_\psi)f_{\varphi(a)}\|_{\beta}  +\|(uC_\varphi-vC_\psi)g_{\varphi, \psi, a}\|_\beta.
\end{eqnarray*}

\medskip

(ii). The proof of (ii) is similar as (i), by interchanging the role of $\varphi$ and $\psi$, and hence we omit it here.

\medskip

(iii). By Lemma \ref{lem01}, we have
\begr
&&\|(uC_\varphi-vC_\psi)f_{\varphi(a)}\|_\beta \ge \left |\calD_{u, \varphi}(a) - \frac{(1-|\varphi(a)|^2)^\alpha(1-|\psi(a)|^2)^\alpha}{|1-\langle \psi(a), \varphi(a) \rangle|^{2\alpha}}\calD_{v, \psi}(a)\right| \nonumber \\
&\ge& |\calD_{u, \varphi}(a)-\calD_{v, \psi}(a)|-\left|1-  \frac{(1-|\varphi(a)|^2)^\alpha(1-|\psi(a)|^2)^\alpha}{|1-\langle \psi(a), \varphi(a) \rangle|^{2\alpha}}\right| |\calD_{v, \psi}(a)|\nonumber \\
&= & |\calD_{u, \varphi}(a)-\calD_{v, \psi}(a)|-\big| (1-|\varphi(a)|^2)^\alpha f_{\varphi(a)}(\varphi(a))\nonumber\\
&& \quad \quad \quad \quad \quad \quad \quad \quad \quad \quad  \quad \quad -(1-|\psi(a)|^2)^\alpha f_{\varphi(a)}(\psi(a)) \big | |\calD_{v, \psi}(a)| \nonumber\\
&\ge& |\calD_{u, \varphi}(a)-\calD_{v, \psi}(a)|-\flat_\alpha(\varphi(a), \psi(a)) |\calD_{v, \psi}(a)| \nonumber\\
&\gtrsim & |\calD_{u, \varphi}(a)-\calD_{v, \psi}(a)|-|\calD_{v, \psi}(a)|\rho(\varphi(a), \psi(a)).\nonumber
\endr
Thus, by (ii), we have
\begin{eqnarray*}
 |\calD_{u, \varphi}(a)-\calD_{v, \psi}(a)|%
&\lesssim&\|(uC_\varphi-vC_\psi)f_{\varphi(a)}\|_\beta+ |\calD_{v, \psi}(a)|\rho(\varphi(a), \psi(a)) \\
&\le&  \|(uC_\varphi-vC_\psi)f_{\varphi(a)}\|_{\beta} + \|(uC_\varphi-vC_\psi)f_{\psi(a)}\|_{\beta}  \\
&& \quad +\|(uC_\varphi-vC_\psi)g_{\psi, \varphi, a}\|_\beta.
\end{eqnarray*}
Interchanging the role of $\varphi$ and $\psi$, it is easy to see that
\begin{eqnarray*}
 |\calD_{u, \varphi}(a)-\calD_{v, \psi}(a)|%
&\le&  \|(uC_\varphi-vC_\psi)f_{\varphi(a)}\|_{\beta} + \|(uC_\varphi-vC_\psi)f_{\psi(a)}\|_{\beta}  \\
&& \quad +\|(uC_\varphi-vC_\psi)g_{\varphi, \psi, a}\|_\beta.
\end{eqnarray*}
Thus, we have
\begin{eqnarray*}
 |\calD_{u, \varphi}(a)-\calD_{v, \psi}(a)|%
&\le&  \|(uC_\varphi-vC_\psi)f_{\varphi(a)}\|_{\beta} + \|(uC_\varphi-vC_\psi)f_{\psi(a)}\|_{\beta}  \\
&& \quad +\min\{\|(uC_\varphi-vC_\psi)g_{\varphi, \psi, a}\|_\beta, \|(uC_\varphi-vC_\psi)g_{\psi, \varphi, a}\|_\beta\}.
\end{eqnarray*}
The proof is complete.
\end{proof}

Next, we introduce the following condition with respect to $\varphi$ and $\psi$: there exists a $C>0$, such that
\begin{equation} \label{eq100}
\inf_{a \in \B}  \frac{1-|\varphi(a)|^2}{|1-\langle \varphi(a), \xi_a \rangle|}>C,
\end{equation}
where $\xi_a=\frac{\Phi_{\varphi(a)}(\psi(a))}{|\Phi_{\varphi(a)}(\psi(a))|}$. We also need its dual version,  namely, there exists a $C>0$, such that
\begin{equation} \label{eq200}
\inf_{a \in \B} \frac{1-|\psi(a)|^2}{|1-\langle \psi(a), \zeta_a \rangle|}>C,
\end{equation}
where $\zeta_a=\frac{\Phi_{\psi(a)}(\varphi(a))}{|\Phi_{\psi(a)}(\varphi(a))|}$.

\begin{lem} \label{lem03}
Let $0<\alpha, \beta<\infty, u, v \in H(\B)$. Let further, $\varphi$ and $\psi$ be holomorphic self-maps of $\B$. Then the following inequalities hold:
\begin{enumerate}
\item[(i).]
$$
\sup_{a \in \B} \|(uC_\varphi-vC_\psi) f_a\|_\beta \lesssim \sup_{j \in \N} \sup_{\xi \in \partial \B} \frac{\|u\langle \varphi, \xi \rangle^j-v\langle \psi, \xi \rangle^j\|_\beta}{\|\langle z, \xi \rangle^j\|_\alpha};
$$
\item[(ii).] Suppose \eqref{eq100} holds, then
\begin{eqnarray*}
\sup_{a \in \B} \|(uC_\varphi-vC_\psi) g_{\varphi, \psi, a} \|_\beta%
&\lesssim& \sup_{j \in \N} \sup_{\xi \in \partial \B} \frac{\|u\langle \varphi, \xi \rangle^j-v\langle \psi, \xi \rangle^j\|_\beta}{\|\langle z, \xi \rangle^j\|_\alpha}\\
&&+ \sup_{j \in \N} \sup_{\xi, \xi' \in \partial \B}\frac{\|u\langle \varphi, \xi \rangle^j\langle \varphi, \xi' \rangle-v\langle \psi, \xi \rangle^j\langle \psi, \xi' \rangle\|_\beta}{\|\langle z, \xi \rangle^j\langle z, \xi' \rangle \|_\alpha};
\end{eqnarray*}
\item[(iii).] Suppose \eqref{eq200} holds, then
\begin{eqnarray*}
\sup_{a \in \B} \|(uC_\varphi-vC_\psi) g_{\psi, \varphi, a} \|_\beta%
&\lesssim& \sup_{j \in \N} \sup_{\xi \in \partial \B} \frac{\|u\langle \varphi, \xi \rangle^j-v\langle \psi, \xi \rangle^j\|_\beta}{\|\langle z, \xi \rangle^j\|_\alpha}\\
&&+ \sup_{j\in \N} \sup_{\xi, \xi' \in \partial \B}\frac{\|u\langle \varphi, \xi \rangle^j\langle \varphi, \xi' \rangle-v\langle \psi, \xi \rangle^j\langle \psi, \xi' \rangle\|_\beta}{\|\langle z, \xi \rangle^j\langle z, \xi' \rangle \|_\alpha}.
\end{eqnarray*}
\end{enumerate}
\end{lem}

\begin{proof}
(i). If $a=0$, then $f_a(z)=1$, and hence
$$
\|(uC_\varphi-vC_\psi)f_a\|_\beta=\|u-v\|_\beta \le \sup_{j \in \N} \sup_{\xi \in \partial \B}  \frac{\|u \langle \varphi, \xi \rangle^j-v\langle \psi, \xi \rangle^j\|_\beta}{\|\langle z, \xi \rangle^j\|_\alpha},
$$
where in the above inequality, we simply consider the case $j=0$ and use the fact that $\|1\|_\alpha=1$.

For any $a \in \B$ with $a \neq 0$, we have
\begin{eqnarray*}
f_a(z)%
&=&\frac{(1-|a|^2)^\alpha}{(1-\langle z, a\rangle)^{2\alpha}}=(1-|a|^2)^\alpha \sum_{k=0}^\infty \frac{\Gamma(k+2\alpha)}{k! \Gamma(2\alpha)} \langle z, a \rangle^k \\
&=&  (1-|a|^2)^\alpha \sum_{k=0}^\infty \frac{\Gamma(k+2\alpha)}{k! \Gamma(2\alpha)} |a|^k  \langle z, \frac{a}{|a|}  \rangle^k.
\end{eqnarray*}
Moreover, for each $\xi \in \partial \B$, it is easy to see
\begin{equation} \label{eq003}
\|\langle z, \xi \rangle^k\|_\alpha=\sup_{z \in \B} |\langle z, \xi \rangle|^k(1-|z|^2)^\alpha \simeq k^{-\alpha},
\end{equation}
uniformly in $\xi$. Thus,
\begin{eqnarray*}
&&\|(uC_\varphi-vC_\psi)f_a\|_\beta\\
&\le& (1-|a|^2)^\alpha \sum_{k=0}^\infty  \frac{\Gamma(k+2\alpha)}{k! \Gamma(2\alpha)} |a|^k \left\|u  \langle \varphi, \frac{a}{|a|} \rangle^k-v  \langle \psi, \frac{a}{|a|}  \rangle^k\right\|_\beta \\
&=& (1-|a|^2)^\alpha \sum_{k=0}^\infty  \frac{\Gamma(k+2\alpha)}{k! \Gamma(2\alpha)} k^{-\alpha}|a|^k k^\alpha \left\|u  \langle \varphi, \frac{a}{|a|} \rangle^k-v  \langle \psi, \frac{a}{|a|}  \rangle^k\right\|_\beta \\
&\le& (1-|a|^2)^\alpha \sum_{k=0}^\infty  \frac{\Gamma(k+2\alpha)}{k! \Gamma(2\alpha)} k^{-\alpha} |a|^k \cdot  \sup_{j \in \N} \sup_{\xi \in \partial \B} \frac{\|u\langle \varphi, \xi \rangle^j-v\langle \psi, \xi \rangle^j\|_\beta}{\|\langle z, \xi \rangle\|_\alpha} \\
&\lesssim& \sup_{j \in \N} \sup_{\xi \in \partial \B} \frac{\|u\langle \varphi, \xi \rangle^j-v\langle \psi, \xi \rangle^j\|_\beta}{\|\langle z, \xi \rangle\|_\alpha},
\end{eqnarray*}
where in the last inequality, we use the fact that
$$
\sum_{k=0}^\infty \frac{\Gamma(k+2\alpha)}{k! \Gamma(2\alpha)} k^{-\alpha}|a|^k \simeq (1-|a|^2)^{-\alpha}.
$$
The desired result follows from by talking supremum of $a$ over $\B$.

\medskip

(ii). First, we observe that for any $\xi, \xi' \in \partial \B$ and $k \in \N$,
$$
\|\langle z, \xi \rangle^k \langle z, \xi' \rangle\|_\alpha \lesssim k^{-\alpha},
$$
which implies
\begin{equation} \label{eq002}
1 \lesssim \frac{k^{-\alpha}}{\|\langle z, \xi \rangle^k \langle z, \xi' \rangle\|_\alpha}.
\end{equation}
Indeed, the above claim follows from \eqref{eq003} and the fact that $\|\langle z, \xi \rangle^k \langle z, \xi' \rangle\|_\alpha \le \|\langle z, \xi \rangle^k\|_\alpha$.

Take and fix some $a \in \B$. Again, without the loss of generality, we may assume that $\Phi_{\varphi(a)}(\psi(a)) \neq 0$, otherwise the statement is trivial.

We may also assume that $\varphi(a) \neq 0$. Indeed, if $\varphi(a)=0$, we have $g_{\varphi, \psi, a}(z)=-\langle z, \xi_a\rangle$, where $\xi_a=\frac{\Phi_{\varphi(a)}(\psi(a))}{|\Phi_{\varphi(a)}(\psi(a))|} \in \partial \B$. Then it is clear that
\begin{eqnarray*}
\|(uC_\varphi-vC_\psi)g_{\varphi, \psi, a}\|_\beta &=&\|u\langle \varphi, \xi_a \rangle-v\langle \psi, \xi_a \rangle\|_\beta \\
&\lesssim& \sup_{j \in \N} \sup_{\xi \in \partial \B} \frac{\|u\langle \varphi, \xi \rangle^j-v\langle \psi, \xi \rangle^j\|_\beta}{\|\langle z, \xi \rangle^j\|_\alpha} .
\end{eqnarray*}
By \cite[Lemma 1.3]{zhu}, we have
\begin{eqnarray*}
\frac{\langle \Phi_{\varphi(a)}(z), \Phi_{\varphi(a)}(\psi(a)) \rangle}{|\Phi_{\varphi(a)}(\psi(a))|}%
&=&\langle \Phi_{\varphi(a)}(z), \xi_a \rangle \\
&=& 1+ \left(\langle \Phi_{\varphi(a)}(z), \xi_a \rangle-1\right) \\
&=& 1- \left( 1-\langle \Phi_{\varphi(a)}(z), \Phi_{\varphi(a)} (\Phi_{\varphi(a)} (\xi_a)) \rangle \right) \\
&=& 1- \frac{(1-|\varphi(a)|^2)(1-\langle z, \xi_a' \rangle)}{(1-\langle z, \varphi(a) \rangle)(1-\langle \varphi(a),  \xi_a' \rangle)},
\end{eqnarray*}
where $\xi_a': = \Phi_{\varphi(a)} (\xi_a) \in \partial \B$. Note that another application of \cite[Lemma 1.3]{zhu} gives
$$
1-\langle \varphi(a), \xi'_a \rangle= \frac{1-|\varphi(a)|^2}{1-\langle \varphi(a), \xi_a \rangle}.
$$

Thus, for each $z \in \B$, we have
\begin{eqnarray*}
g_{\varphi, \psi, a}(z)%
&=& f_{\varphi(a)}(z) \cdot \frac{\langle \Phi_{\varphi(a)}(z), \Phi_{\varphi(a)}(\psi(a)) \rangle}{|\Phi_\varphi(a)(\psi(a))|} \\
&=& f_{\varphi(a)}(z)- f_{\varphi(a)}(z) \cdot  \frac{(1-|\varphi(a)|^2)(1-\langle z, \xi_a' \rangle)}{(1-\langle z, \varphi(a) \rangle)(1-\langle \varphi(a),  \xi_a' \rangle)},
\end{eqnarray*}
which implies
\begin{eqnarray*}
\|(uC_\varphi-vC_\psi)g_{\varphi, \psi, a}\|_\beta
&\le& \|(uC_\varphi-vC_\psi)(f_{\varphi(a)})\|_\beta\\
&& + \left\|(uC_\varphi-vC_\psi)\left(f_{\varphi(a)} \cdot  \frac{(1-|\varphi(a)|^2)(1-\langle z, \xi_a' \rangle)}{(1-\langle z, \varphi(a) \rangle)(1-\langle \varphi(a),  \xi_a' \rangle)}\right)\right\|_\beta
\end{eqnarray*}
By part (i), the first term is bounded by
$$
\sup_{j\in \N} \sup_{\xi \in \partial \B}  \frac{\|u\langle \varphi, \xi \rangle^j-v\langle \psi, \xi \rangle^j\|_\beta}{\|\langle z, \xi \rangle\|_\alpha}.
$$
To bound the second term, we first note that
\begin{eqnarray*}
 &&f_{\varphi(a)}(z) \cdot  \frac{(1-|\varphi(a)|^2)(1-\langle z, \xi_a' \rangle)}{(1-\langle z, \varphi(a) \rangle)(1-\langle \varphi(a),  \xi_a' \rangle)}\\
&= & \left[(1-|\varphi(a)|^2)^{\alpha} \sum_{k=0}^\infty \frac{\Gamma(k+2\alpha)}{k!\Gamma(2\alpha)} \langle z, \varphi(a) \rangle^k\right] \cdot \left[ \frac{1-|\varphi(a)|^2}{1-\langle \varphi(a), \xi_a' \rangle} \cdot (1-\langle z, \xi'_a \rangle ) \cdot \sum_{k=0}^\infty \langle z, \varphi(a) \rangle^k\right] \\
&= &J_1(z)-J_2(z),
\end{eqnarray*}
where both $J_1(z)$ and $J_2(z)$ are holomorphic on $\B$, defined by
\begin{eqnarray*}
J_1(z)%
&:=& \frac{(1-|\varphi(a)|^2)^{1+\alpha}}{1-\langle \varphi(a), \xi_a' \rangle} \left[ \sum_{k=0}^\infty \frac{\Gamma(k+2\alpha)}{k!\Gamma(2\alpha)} \langle z, \varphi(a) \rangle^k\right] \cdot \left[ \sum_{k=0}^\infty \langle z, \varphi(a) \rangle^k\right] \\
&=&   \frac{(1-|\varphi(a)|^2)^{1+\alpha}}{1-\langle \varphi(a), \xi_a' \rangle} \cdot \sum_{k=0}^\infty \left( \sum_{j=0}^k \frac{\Gamma(j+2\alpha)}{j! \Gamma(2\alpha)} \right) \langle z, \varphi(a) \rangle^k
\end{eqnarray*}
and
$$
J_2(z):= \frac{(1-|\varphi(a)|^2)^{1+\alpha}}{1-\langle \varphi(a), \xi_a' \rangle} \cdot \sum_{k=0}^\infty \left( \sum_{j=0}^k \frac{\Gamma(j+2\alpha)}{j! \Gamma(2\alpha)} \right) \langle z, \varphi(a) \rangle^k \langle z, \xi'_a \rangle.
$$
Hence,
\begin{eqnarray*}
&& \left\|(uC_\varphi-vC_\psi)\left(f_{\varphi(a)} \cdot  \frac{(1-|\varphi(a)|^2)(1-\langle z, \xi_a' \rangle)}{(1-\langle z, \varphi(a) \rangle)(1-\langle \varphi(a),  \xi_a' \rangle)}\right)\right\|_\beta\nonumber\\
&\le& \|(uC_\varphi-vC_\psi)J_1\|_\beta  +\|(uC_\varphi-vC_\psi)J_2\|_\beta
\end{eqnarray*}

\medskip
 By Stirling's formula, we have
$$
 \sum_{j=0}^k \frac{\Gamma(j+2\alpha)}{j! \Gamma(2\alpha)} \simeq \sum_{j=0}^k j^{2\alpha-1} \simeq k^{2\alpha}
$$
for $k$ large enough. Thus, by \eqref{eq100} and \eqref{eq003}, we have
\begin{eqnarray*}
&&\|(uC_\varphi-vC_\psi)J_1\|_\beta\\ %
&\lesssim& (1-|\varphi(a)|^2)^{1+\alpha} \sum_{k=0}^\infty \left( \sum_{j=0}^k \frac{\Gamma(j+2\alpha)}{j! \Gamma(2\alpha)} \right) \left\|(uC_\varphi-vC_\psi)\langle z, \varphi(a) \rangle^k\right\|_\beta \\
&\lesssim& (1-|\varphi(a)|^2)^{1+\alpha} \sum_{k=1}^\infty k^{2\alpha} |\varphi(a)|^k \left\|(uC_\varphi-vC_\psi)\Big(  \big\langle z, \frac{\varphi(a)}{|\varphi(a)|}  \big \rangle^k \Big)\right\|_\beta \\
&\lesssim& (1-|\varphi(a)|^2)^{1+\alpha} \sum_{k=1}^\infty \frac{1}{k^\alpha} \cdot k^{2\alpha}|\varphi(a)|^k \cdot \frac{\left\|(uC_\varphi-vC_\psi)\Big( \big\langle z, \frac{\varphi(a)}{|\varphi(a)|} \big\rangle^k \Big)\right\|_\beta}{\left\|  \big \langle z, \frac{\varphi(a)}{|\varphi(a)|} \big \rangle^k\right\|_\alpha} \\
&\lesssim& \sup_{j\in \N} \sup_{\xi \in \partial \B}  \frac{\|u\langle \varphi, \xi \rangle^j-v\langle \psi, \xi \rangle^j\|_\beta}{\|\langle z, \xi \rangle\|_\alpha}.
\end{eqnarray*}

The estimation of $\|(uC_\varphi-vC_\psi)J_2\|_\beta$ is similar as the previous part, by replacing the role of \eqref{eq003} by \eqref{eq002}, and
$\langle z, \varphi(a) \rangle^k$ by $
\langle z, \varphi(a) \rangle^k \langle z, \xi'_a \rangle.$  Hence, we omit the detail here.
The desired estimation follows from combining the above two estimations.

\medskip

(iii). The proof of (iii) is similar as  (ii), by replacing \eqref{eq100} by it dual version \eqref{eq200}.
\end{proof}

\begin{rem}
In the above lemma, the assumptions \eqref{eq100} and \eqref{eq200} can be dropped when $n=1$. Indeed, by the definition of $\Phi_{\varphi(a)}(z)$, we have
\begin{eqnarray*}
\Phi_{\varphi(a)}(z)%
&=& \left(\varphi(a)-P_{\varphi(a)}(z)-s_{\varphi(a)} Q_{\varphi(a)}(z) \right) \sum_{k=0}^\infty \langle z, \varphi(a) \rangle^k\\
&=& \left(\varphi(a)-P_{\varphi(a)}(z)\right) \sum_{k=0}^\infty \langle z, \varphi(a) \rangle^k,
\end{eqnarray*}
where in the last equality, we use the fact that the orthogonal projection $Q_a$ vanishes when $n=1$. Hence
\begin{eqnarray*}
\frac{\langle \Phi_{\varphi(a)}(z), \Phi_{\varphi(a)}(\psi(a)) \rangle}{|\Phi_{\varphi(a)}(\psi(a))|}%
&=& \langle \Phi_{\varphi(a)}(z), \xi_a \rangle\\
&=& \frac{|\varphi(a)|^2-1}{|\varphi(a)|^2} \cdot \langle \varphi(a), \xi_a \rangle \cdot \sum_{k=0}^\infty \langle z, \varphi(a) \rangle^k.
\end{eqnarray*}

The desired claim then follows from a similar argument as the one in Lemma \ref{lem03}, (ii) and an application of Cauchy's inequality. Hence we omit the detail here.
\end{rem}

\begin{thm} \label{boundedness}
Let $0<\alpha, \beta<\infty, u, v \in H(\B)$. Let further, $\varphi$ and $\psi$ be holomorphic self-maps of $\B$ satisfying \eqref{eq100} or \eqref{eq200}. Then $uC_\varphi-vC_\psi: H_\alpha^\infty \mapsto H_\beta^\infty$ is bounded if and only if
\begin{equation} \label{eq001}
\sup_{j \in \N} \sup_{\xi \in \partial \B} \frac{\|u\langle \varphi, \xi \rangle^j-v\langle \psi, \xi \rangle^j\|_\beta}{\|\langle z, \xi \rangle^j\|_\alpha} <\infty
\end{equation}
and
\begin{equation} \label{eq0011}
\sup_{j\in \N} \sup_{\xi, \xi' \in \partial \B}\frac{\|u\langle \varphi, \xi \rangle^j\langle \varphi, \xi' \rangle-v\langle \psi, \xi \rangle^j\langle \psi, \xi' \rangle\|_\beta}{\|\langle z, \xi \rangle^j\langle z, \xi' \rangle \|_\alpha}<\infty.
\end{equation}
\end{thm}

\begin{proof}
\textbf{Necessity.} Suppose $uC_\varphi-vC_\psi$ is bounded. For any $j \in \N$ and $\xi, \xi' \in \partial \B$, consider the probe funcitons
$$
f_{j, \xi}(z)=\frac{\langle z, \xi \rangle^j}{\|\langle z, \xi \rangle^j\|_\alpha} \quad \textrm{and} \quad f_{j, \xi, \xi'}(z)=\frac{\langle z, \xi \rangle^j\langle z, \xi' \rangle}{\|\langle z, \xi \rangle^j \langle z, \xi' \rangle \|_\alpha}
$$
Then $\|f_{j, \xi}\|_{\alpha}=\|f_{j, \xi, \xi'}\|_\alpha=1$. Thus, by the boundedness of $uC_\varphi-vC_\psi$, we have
$$
\frac{\|u\langle \varphi, \xi \rangle^j-v\langle \psi, \xi \rangle^j\|_\beta}{\|\langle z, \xi \rangle^j\|_\alpha} =\|(uC_\varphi-vC_\psi) f_{j, \xi}\|_\beta \le \|uC_\varphi-vC_\psi\|<\infty,
$$
and
$$
\frac{\|u\langle \varphi, \xi \rangle^j\langle \varphi, \xi' \rangle-v\langle \psi, \xi \rangle^j\langle \psi, \xi' \rangle\|_\beta}{\|\langle z, \xi \rangle^j\langle z, \xi' \rangle \|_\alpha}=\|(uC_\varphi-vC_\psi) f_{j, \xi, \xi'} \|_\beta \le \|uC_\varphi-vC_\psi\|<\infty,
$$

The desired result then follows by take the supremum of $j$, $\xi$ and $\xi'$ on both sides of the above inequalities.

\medskip

\textbf{Sufficiency.} Suppose \eqref{eq001} and \eqref{eq0011} holds. Moreover, without the loss of generality, we assume \eqref{eq100} holds. Then for any $f \in H^\infty_\alpha$ with $\|f\|_{H^\infty_\alpha} \le 1$ and using Lemma \ref{lem01},  we have
\begin{eqnarray*}
&&\|(uC_\varphi-vC_\psi)f\|_\beta\\
&= &\sup_{z \in \B} \left| u(z)f(\varphi(z))-v(z)f(\psi(z)) \right| (1-|z|^2)^\beta \\
&\le& \sup_{z \in \B} \left| f(\varphi(z))(1-|\varphi(z)|^2)^\alpha-f(\psi(z))(1-|\psi(z)|^2)^\alpha \right| |\calD_{u, \varphi}(z)| \\
&& \quad \quad +\sup_{z \in \B} |f(\psi(z))|(1-|\psi(z)|^2)^\alpha |\calD_{u, \varphi}(z)-\calD_{v, \psi}(z)| \\
&\le&\sup_{z \in \B} \flat_\alpha(\varphi(z), \psi(z)) |\calD_{u, \varphi}(z)|+\sup_{z \in \B} |\calD_{u, \varphi}(z)-\calD_{v, \psi}(z)| \\
& \lesssim& \sup_{z \in \B}|\calD_{u, \varphi}(z)| \rho(\varphi(z), \psi(z))+\sup_{z \in \B} |\calD_{u, \varphi}(z)-\calD_{v, \psi}(z)|.
\end{eqnarray*}
Hence, by Lemma \ref{lem02} and Lemma \ref{lem03}, we have
\begin{eqnarray*}
&&\|(uC_\varphi-vC_\psi)f\|_\beta\\
&\lesssim& \sup_{a \in \B} \|(uC_\varphi-vC_\psi)f_{\varphi(a)}\|_{\beta}+\sup_{a \in \B} \sup_{a \in \B} \|(uC_\varphi-vC_\psi)f_{\psi(a)}\|_{\beta}\\
&  & +  \sup_{a \in \B} \|(uC_\varphi-vC_\psi)g_{\varphi, \psi, a}\|_\beta\\
&\lesssim& \sup_{j\in \N} \sup_{\xi \in \partial \B} \frac{\|u\langle \varphi, \xi \rangle^j-v\langle \psi, \xi \rangle^j\|_\beta}{\|\langle z, \xi \rangle^j\|_\alpha} + \sup_{j \in \N} \sup_{\xi, \xi' \in \partial \B}\frac{\|u\langle \varphi, \xi \rangle^j\langle \varphi, \xi' \rangle-v\langle \psi, \xi \rangle^j\langle \psi, \xi' \rangle\|_\beta}{\|\langle z, \xi \rangle^j\langle z, \xi' \rangle \|_\alpha}\\
&<&\infty.
\end{eqnarray*}
Therefore, $uC_\varphi-vC_\psi: H^\infty_\alpha \mapsto H^\infty_\beta$ is bounded. The proof is complete.

\end{proof}


\section{Essential norm estimates. }

In this section, we give an estimate for the essential norm of $uC_\varphi-vC_\psi$ from $H^\infty_\alpha$ to $H^\infty_\beta$.

 Recall that the essential norm $\|T\|_{e,\XX\rightarrow \YY}$ of a bounded linear operator $T:\XX\rightarrow \YY$ is defined as the distance from $T$ to the set of compact operators $K$ mapping $\XX$ into $\YY$, that is, $\|T\|_{e, \XX\rightarrow \YY}=\inf\{\|T-K\|_{\XX\rightarrow \YY}: K~\mbox{
is compact}~~\},$ where $\|\cdot\|_{\XX\rightarrow \YY}$ is the operator norm.\msk

\begin{lem} \label{lem04}
Let $0<\alpha, \beta<\infty, u, v \in H(\B)$. Let further, $\varphi$ and $\psi$ be holomorphic self-maps of $\B$. Then the following inequalities hold:
\begin{enumerate}
\item[(i).]
$$
\limsup_{|a| \to 1} \|(uC_\varphi-vC_\psi) f_a\|_\beta \lesssim \limsup_{j \to \infty} \sup_{\xi \in \partial \B} \frac{\|u\langle \varphi, \xi \rangle^j-v\langle \psi, \xi \rangle^j\|_\beta}{\|\langle z, \xi \rangle^j\|_\alpha};
$$
\item[(ii).] Suppose \eqref{eq100} holds, then
\begin{eqnarray*}
\limsup_{|\varphi(a)| \to 1} \|(uC_\varphi-vC_\psi) g_{\varphi, \psi, a} \|_\beta%
&\lesssim& \limsup_{j \to \infty} \sup_{\xi \in \partial \B} \frac{\|u\langle \varphi, \xi \rangle^j-v\langle \psi, \xi \rangle^j\|_\beta}{\|\langle z, \xi \rangle^j\|_\alpha} \\
&&+ \limsup_{j \to \infty} \sup_{\xi, \xi' \in \partial \B}\frac{\|u\langle \varphi, \xi \rangle^j\langle \varphi, \xi' \rangle-v\langle \psi, \xi \rangle^j\langle \psi, \xi' \rangle\|_\beta}{\|\langle z, \xi \rangle^j\langle z, \xi' \rangle \|_\alpha};
\end{eqnarray*}
\item[(iii).] Suppose \eqref{eq200} holds, then
\begin{eqnarray*}
\limsup_{|\psi(a)| \to 1} \|(uC_\varphi-vC_\psi) g_{\psi, \varphi, a} \|_\beta%
&\lesssim&  \limsup_{j\to \infty} \sup_{\xi \in \partial \B} \frac{\|u\langle \varphi, \xi \rangle^j-v\langle \psi, \xi \rangle^j\|_\beta}{\|\langle z, \xi \rangle^j\|_\alpha} \\
&&+ \limsup_{j\to \infty} \sup_{\xi, \xi' \in \partial \B}\frac{\|u\langle \varphi, \xi \rangle^j\langle \varphi, \xi' \rangle-v\langle \psi, \xi \rangle^j\langle \psi, \xi' \rangle\|_\beta}{\|\langle z, \xi \rangle^j\langle z, \xi' \rangle \|_\alpha}.
\end{eqnarray*}
\end{enumerate}
\end{lem}

\begin{proof}  (i). For each $N \in \N$ and $a \in \B$ with $a \neq 0$, the proof of (i) in Lemma \ref{lem03} gives
\begin{eqnarray*}
&&\|(uC_\varphi-vC_\psi) f_a\|_\beta\\
&\lesssim& (1-|a|^2)^\alpha \sum_{k=0}^N  \frac{\Gamma(k+2\alpha)}{k! \Gamma(2\alpha)} |a|^k \left\| u   \langle \varphi, \frac{a}{|a|}   \rangle^k- v  \langle \psi, \frac{a}{|a|}  \rangle^k \right\|_\beta \\
&& \quad +(1-|a|^2)^\alpha \sum_{N+1}^\infty \frac{\Gamma(k+2\alpha)}{k! \Gamma(2\alpha)} k^{-\alpha} |a|^k \sup_{j \ge N+1} \sup_{\xi \in \partial \B}   \frac{\|u\langle \varphi, \xi \rangle^j-v\langle \psi, \xi \rangle^j\|_\beta}{\|\langle z, \xi \rangle^j\|_\alpha}.
\end{eqnarray*}
Now for this fixed $N$, by letting $|a| \to 1$, we have
$$
\limsup_{|a| \to 1} \|(uC_\varphi-vC_\psi) f_a\|_\beta \lesssim \sup_{j \ge N+1} \sup_{\xi \in \partial \B}   \frac{\|u\langle \varphi, \xi \rangle^j-v\langle \psi, \xi \rangle^j\|_\beta}{\|\langle z, \xi \rangle^j\|_\alpha},
$$
which implies
$$
\limsup_{|a| \to 1} \|(uC_\varphi-vC_\psi) f_a\|_\beta \lesssim \limsup_{j \to \infty} \sup_{\xi \in \partial \B} \frac{\|u\langle \varphi, \xi \rangle^j-v\langle \psi, \xi \rangle^j\|_\beta}{\|\langle z, \xi \rangle^j\|_\alpha}.
$$

(ii).  Again, from the proof of Lemma \ref{lem03}, we have for each $N \in \N$,
\begin{eqnarray*}
&& \|(uC_\varphi-vC_\psi)(g_{\varphi, \psi, a}-f_{\varphi(a)})\|_\beta\\
&\lesssim & (1-|\varphi(a)|^2)^{1+\alpha} \sum_{k=1}^N k^{2\alpha} |\varphi(a)|^k \left\| (uC_\varphi-vC_\psi) \Big(  \langle z, \frac{\varphi(a)}{|\varphi(a)|}   \rangle^k \Big ) \right\|_\beta \\
&& +(1-|\varphi(a)|^2)^{1+\alpha} \sum_{k=1}^N k^{2\alpha} |\varphi(a)|^k \left\| (uC_\varphi-vC_\psi) \Big(  \langle z, \frac{\varphi(a)}{|\varphi(a)|}   \rangle^k \langle z, \varphi(a) \rangle^k \langle z, \xi'_a \rangle\Big) \right\|_\beta \\
&&  + (1-|\varphi(a)|^2)^{1+\alpha} \sum_{k=N+1}^\infty k^{\alpha} |\varphi(a)|^k \sup_{k \ge N} \sup_{\xi \in \partial \B} \frac{\left\| (uC_\varphi-vC_\psi) \Big(   \langle z, \xi   \rangle^k \Big ) \right\|_\beta}{\left\| \langle z, \xi  \rangle \right\|_\alpha} \\
&& + (1-|\varphi(a)|^2)^{1+\alpha} \sum_{k=N+1}^\infty k^{\alpha} |\varphi(a)|^k \sup_{k \ge N} \sup_{\xi, \xi' \in \partial \B} \frac{\left\| (uC_\varphi-vC_\psi) \left(  \langle z, \xi   \rangle^k \langle z, \xi' \rangle\right ) \right\|_\beta}{\left\| \left\langle z, \xi \right\rangle ^k \langle z, \xi' \rangle \right\|_\alpha}.
\end{eqnarray*}
Let $|\varphi(a)| \to 1$, then we have
\begin{eqnarray*}
\limsup_{|\varphi(a)| \to 1} \|(uC_\varphi-vC_\psi)(g_{\varphi, \psi, a}-f_{\varphi(a)})\|_\beta%
&\lesssim& \sup_{k \ge N} \sup_{\xi \in \partial \B} \frac{\left\| (uC_\varphi-vC_\psi)  \Big( \left \langle z, \xi \right \rangle^k \Big ) \right\|_\beta}{\left\| \left\langle z, \xi \right\rangle \right\|_\alpha} \\
&& + \sup_{k \ge N} \sup_{\xi, \xi' \in \partial \B} \frac{\left\| (uC_\varphi-vC_\psi)\Big( \left \langle z, \xi \right \rangle^k \langle z, \xi' \rangle\Big ) \right\|_\beta}{\left\| \left\langle z, \xi \right\rangle ^k \langle z, \xi' \rangle \right\|_\alpha},
\end{eqnarray*}
which clearly implies the desired estimate.

\medskip

(iii). The proof for (iii) again is similar as (ii), and hence we omit it here.
\end{proof}

\begin{thm}
Let $0<\alpha, \beta<\infty, u, v \in H(\B)$. Let further, $\varphi$ and $\psi$ be holomorphic self-maps of $\B$ satisfying \eqref{eq100} and \eqref{eq200}. Suppose that $uC_\varphi: H^\infty_\alpha \mapsto H^\infty_\beta$ and  $uC_\varphi: H^\infty_\alpha \mapsto H^\infty_\beta$ are bounded, then
\begin{eqnarray*}
\|uC_\varphi-vC_\psi\|_{e, H^\infty_\alpha \mapsto H^\infty_\beta}%
&\simeq& \limsup_{j \to \infty}  \sup_{\xi \in \partial \B}\frac{\|u\langle \varphi, \xi \rangle^j-v\langle \psi, \xi \rangle^j\|_\beta}{\|\langle z, \xi \rangle^j\|_\alpha} \\
&&+  \limsup_{j \to \infty} \sup_{\xi, \xi' \in \partial \B}\frac{\|u\langle \varphi, \xi \rangle^j\langle \varphi, \xi' \rangle-v\langle \psi, \xi \rangle^j\langle \psi, \xi' \rangle\|_\beta}{\|\langle z, \xi \rangle^j\langle z, \xi' \rangle \|_\alpha}.
\end{eqnarray*}
\end{thm}

\begin{proof}
Following a standard argument (see, e.g., \cite[Theorem 3.1]{HLS}), one can easily establish that
\begin{eqnarray*}
\|uC_\varphi-vC_\psi\|_{e, H^\infty_\alpha \mapsto H^\infty_\beta}%
&\lesssim& \lim_{r \to 1} \sup_{|\varphi(z)|>r} |\calD_{u,\varphi}(z)|\rho(\varphi(z), \psi(z))\\
&& +\lim_{r \to 1} \sup_{|\psi(z)|>r} |\calD_{v,\psi}(z)|\rho(\varphi(z), \psi(z))\\
&&+ \lim_{r \to 1} \sup_{|\varphi(z)|, |\psi(z)|>r} |\calD_{u, \varphi}(z)-\calD_{v, \psi}(z)|.
\end{eqnarray*}
Thus, by Lemmas \ref{lem02} and   \ref{lem04}, we have
\begin{eqnarray*}
&& \|uC_\varphi-vC_\psi\|_{e, H^\infty_\alpha \mapsto H^\infty_\beta}\\%
&\lesssim&  \limsup_{|\varphi(a)| \to 1} \|(uC_\varphi-vC_\psi)f_{\varphi(a)}\|_\beta  + \limsup_{|\psi(a)| \to 1} \|(uC_\varphi-vC_\psi)f_{\psi(a)}\|_\beta \\
&& + \limsup_{|\varphi(a)| \to 1} \|(uC_\varphi-vC_\psi)g_{\varphi, \psi, a}\|_\beta   + \limsup_{|\psi(a)| \to 1} \|(uC_\varphi-vC_\psi)g_{\psi, \varphi, a}\|_\beta\\
& \lesssim& \limsup_{j\to \infty} \sup_{\xi \in \partial \B} \frac{\|u\langle \varphi, \xi \rangle^j-v\langle \psi, \xi \rangle^j\|_\beta}{\|\langle z, \xi \rangle^j\|_\alpha} \\
&& + \limsup_{j \to \infty} \sup_{\xi, \xi' \in \partial \B}\frac{\|u\langle \varphi, \xi \rangle^j\langle \varphi, \xi' \rangle-v\langle \psi, \xi \rangle^j\langle \psi, \xi' \rangle\|_\beta}{\|\langle z, \xi \rangle^j\langle z, \xi' \rangle \|_\alpha}.
\end{eqnarray*}

Next, we shall prove that
\begin{eqnarray*}
\|uC_\varphi-vC_\psi\|_{e, H^\infty_\alpha \mapsto H^\infty_\beta}%
&\gtrsim&  \limsup_{j\to \infty} \sup_{\xi \in \partial \B} \frac{\|u\langle \varphi, \xi \rangle^j-v\langle \psi, \xi \rangle^j\|_\beta}{\|\langle z, \xi \rangle^j\|_\alpha} \\
&& + \limsup_{j \to \infty} \sup_{\xi, \xi' \in \partial \B}\frac{\|u\langle \varphi, \xi \rangle^j\langle \varphi, \xi' \rangle-v\langle \psi, \xi \rangle^j\langle \psi, \xi' \rangle\|_\beta}{\|\langle z, \xi \rangle^j\langle z, \xi' \rangle \|_\alpha}.
\end{eqnarray*}
To see this, recall the test functions $f_{j, \xi}$ and $f_{j, \xi, \xi'}$ defined in the proof of Theorem \ref{boundedness}, namely
for $j \in \N$ and $\xi, \xi' \in \partial \B$, we write
$$
f_{j, \xi}(z)=\frac{\langle z, \xi \rangle^j}{\|\langle z, \xi \rangle^j \|_\alpha} \quad \textrm{and} \quad f_{j, \xi, \xi'}(z)=\frac{\langle z, \xi \rangle^j \langle z, \xi' \rangle}{\|\langle z, \xi \rangle^j \langle z, \xi' \rangle\|_\alpha}.
$$
We first show that
\begin{equation} \label{eq006}
\|uC_\varphi-vC_\psi\|_{e, H^\infty_\alpha \mapsto H^\infty_\beta} \ge  \limsup_{j \to \infty} \sup_{\xi \in \partial \B} \frac{\|u\langle \varphi, \xi \rangle^j-v\langle \psi, \xi \rangle^j\|_\beta}{\|\langle z, \xi \rangle^j\|_\alpha} .
\end{equation}
Since $\|f_{j, \xi}\|_\alpha=1$ for all $j\in \N$ and $\xi \in \partial \B$, by the boundedness assumption on both $uC_\varphi$ and $vC_\psi$, we have
$$
\sup_{j\in \N} \sup_{\xi \in \partial \B} \|(uC_\varphi-vC_\psi)(f_{j, \xi})\|_\beta<\infty.
$$
Take and fix any $\varepsilon>0$, then for each $k \in \N$, we can take a $\xi_k \in \partial \B$, such that
$$
\|(uC_\varphi-vC_\psi)(f_{k, \xi_k})\|_\beta \ge \sup_{\xi \in \partial \B} \| (uC_\varphi-vC_\psi)(f_{k, \xi})\|_\beta-\varepsilon.
$$
Now we consider the set $\{f_{k, \xi_k}\}_{k \in \N}$. It is clear that  $f_{k, \xi_k} \to 0$  uniformly on compact subsets of $\B$. Hence, if $K$ is any compact operator from $H^\infty_\alpha$ to $H^\infty_\beta$, then
$$
\lim_{k \to \infty} \|Kf_{k, \xi_k}\|_{\beta}=0
$$
Hence
\begin{eqnarray*}
\|uC_\varphi-vC_\psi-K\|_{ H^\infty_\alpha \mapsto H^\infty_\beta}%
&\ge& \limsup_{k \to \infty} \|(uC_\varphi-vC_\psi-K)(f_{k, \xi_k})\|_\beta \\
&=& \limsup_{k \to \infty} \|(uC_\varphi-vC_\psi)(f_{k, \xi_k})\|_\beta \\
&\ge& \limsup_{k \to \infty} \sup_{\xi \in \partial \B} \| (uC_\varphi-vC_\psi)(f_{k, \xi})\|_\beta-\varepsilon,
\end{eqnarray*}
which implies the desired result by first letting $\varepsilon \to 0$ and then taking the infimum of $K$ over all the compact operators.

Similarly, we can show that
$$
\|uC_\varphi-vC_\psi\|_{e, H^\infty_\alpha \mapsto H^\infty_\beta} \ge   \limsup_{j \to \infty} \sup_{\xi, \xi' \in \partial \B}\frac{\|u\langle \varphi, \xi \rangle^j\langle \varphi, \xi' \rangle-v\langle \psi, \xi \rangle^j\langle \psi, \xi' \rangle\|_\beta}{\|\langle z, \xi \rangle^j\langle z, \xi' \rangle \|_\alpha}.
$$
Finally, combining the above estimate with \eqref{eq006}, we get the desired result.
\end{proof}

As a corollary, we have the following result on the compactness of $uC_\varphi-vC_\psi$, which is simply due to the essential norm of a compact operator is $0$.

\begin{cor}
Let $0<\alpha, \beta<\infty, u, v \in H(\B)$. Let further, $\varphi$ and $\psi$ be holomorphic self-maps of $\B$ satisfying \eqref{eq100} and \eqref{eq200}. Suppose that $uC_\varphi: H^\infty_\alpha \mapsto H^\infty_\beta$ and  $uC_\varphi: H^\infty_\alpha \mapsto H^\infty_\beta$ are bounded, then $uC_\varphi-vC_\psi$ is compact if and only if
$$
\limsup_{j \to \infty}  \sup_{\xi \in \partial \B}\frac{\|u\langle \varphi, \xi \rangle^j-v\langle \psi, \xi \rangle^j\|_\beta}{\|\langle z, \xi \rangle^j\|_\alpha} =0
$$
and
$$
 \limsup_{j \to \infty} \sup_{\xi, \xi' \in \partial \B}\frac{\|u\langle \varphi, \xi \rangle^j\langle \varphi, \xi' \rangle-v\langle \psi, \xi \rangle^j\langle \psi, \xi' \rangle\|_\beta}{\|\langle z, \xi \rangle^j\langle z, \xi' \rangle \|_\alpha}=0.
$$
\end{cor}

\end{document}